\documentclass{article}
\usepackage{amsthm,amssymb,amsmath}
\usepackage{color}
\usepackage{epsfig,verbatim}

\newtheorem{thm}{Theorem}
\newtheorem{prop}[thm]{Proposition}
\newtheorem{defn}[thm]{Definition}
\newtheorem{lem}[thm]{Lemma}

\newtheorem{conj}[thm]{Conjecture}
\newtheorem*{thmworstdestab}{Theorem \ref{thm:worst_destab}}
\newtheorem*{propL2lower}{Theorem \ref{prop:L2lower}}
\newtheorem*{thmcalflow}{Theorem \ref{thm:calflow}}

\theoremstyle{remark}
\newtheorem*{rem}{Remark}

\title{Optimal test-configurations for toric varieties}
\author{G\'abor Sz\'ekelyhidi}
\date{}

\begin{document}

\maketitle

\begin{abstract}
  On a K-unstable toric variety we show the existence of an optimal
  destabilising convex function. We show that if this is piecewise linear then
  it gives rise to a decomposition into semistable pieces analogous to the
  Harder-Narasimhan filtration of an unstable vector bundle. We also show that
  if the Calabi flow exists for all time on a toric variety
  then it minimises the Calabi functional. In this case the infimum 
  of the Calabi functional is given by the supremum of
  the normalised Futaki invariants over all destabilising
  test-configurations, as predicted by a conjecture of Donaldson. 
\end{abstract}

\section{Introduction}
The Harder-Narasimhan filtration of an unstable vector bundle is a canonical
filtration with semistable quotient sheaves. It arises for example when
computing the infimum of the Yang-Mills functional (see
Atiyah-Bott~\cite{AB83}), which is analogous to the Calabi functional on
a K\"ahler manifold.
Bruasse and Teleman~\cite{BT05} have shown that the
Harder-Narasimhan filtration arises in other moduli problems as well, when one
looks at the optimal destabilising one-parameter subgroup for a non-semistable
point. The notion of optimal one-parameter subgroups is well known in
geometric invariant theory, see for example Kirwan~\cite{Kir84}.

In the meantime much progress has been made in studying the stabiliy of
manifolds in relation to the existence of canonical metrics. Such a relationship
was originally conjectured by Yau~\cite{Yau93} in the case of K\"ahler-Einstein
metrics. Tian~\cite{Tian97} and Donaldson~\cite{Don01}, \cite{Don02}  made great
progress on this problem, and by now there is a large relevant literature. 
For us the important work is \cite{Don02} through which we have a good
understanding of stability for toric varieties (for further work on toric
varieties see also~\cite{Don05_1},\cite{Don06}). In particular we can
construct a large family of test-configurations, 
which are analogous to one-parameter
subgroups, in terms of data on the moment polytope. In this paper we
use this to study the optimal 
destabilising test-configuration on an unstable toric variety and the
Harder-Narasimhan type decomposition that it gives rise to.

Recall that a compact polarised toric variety $(X,L)$ corresponds to a polytope
$P\subset\mathbf{R}^n$, which is equipped with a canonical measure $d\sigma$ on
the boundary $\partial P$ (for details see Section~\ref{sec:prelim}). 
We also let $d\mu$ denote
the Lebesgue measure on the interior of $P$, and write $\hat{S}$ for the
quotient $Vol(\partial P,d\sigma)/Vol(P,d\mu)$. This is essentially
the average scalar
curvature of metrics on the toric variety. Let us define the functional 
\[ \mathcal{L}(f) = \int_{\partial P} f\, d\sigma - \hat{S}\int_P
f\,d\mu,\]
which by the choice of $\hat{S}$ vanishes on constant functions. 
Donaldson shows that given a rational piecewise linear convex function $f$ on
$P$, one can define a test-configuration for $(X,L)$ with generalised Futaki invariant
$\mathcal{L}(f)$ (if we scale the Futaki invariant in the right way). We will
say that the toric variety is unstable if for some convex function $f$ we have
$\mathcal{L}(f)<0$. The natural
norm for the test-configuration is given by the $L^2$-norm of $f$ at least if we
consider $f$ with zero mean. This means that the optimal destabilising
test-configuration we are looking for in the unstable case
should minimise the functional 
\[ W(f) = \frac{\mathcal{L}(f)}{\Vert f\Vert_{L^2}},\]
defined for non-zero convex functions. Note that the minimum will be
negative and the minimiser automatically has zero mean. 
The space of functions $\mathcal{C}_1$ on which we
minimise is the set of continuous convex functions on $P^*$, integrable on
$\partial P$, where $P^*$ is union of $P$ and its codimension one faces.
Our first result in Section~\ref{sec:worstdestab} is 
\begin{thmworstdestab} Let the toric variety with moment
  polytope $P$ be unstable. Then there exists a convex minimiser
  $\Phi\in\mathcal{C}_1\cap L^2(P)$ for
  $W$ which is unique up
  to scaling. Let us fix the scaling by requiring 
  that 
  \[ \mathcal{L}(\Phi)=-\Vert\Phi\Vert^2_{L^2}.\]
  Letting $B=\hat{S}-\Phi$, we then have $\mathcal{L}_B(f)\geqslant
  0$ for all convex
  functions $f$, and $\mathcal{L}_B(\Phi)=0$. Conversely these two
  conditions characterise $\Phi$.  
\end{thmworstdestab}
\noindent Here we define 
\[ \mathcal{L}_B(f) = \int_{\partial P} f\,d\sigma - \int_P Bf\,d\mu.\]
Note that $\Phi$ would only define a test-configuration if it were piecewise
linear. This is not known and perhaps not true in general
so instead we may think of $\Phi$ as a limit
of test-configurations. The proof is based on a compactness theorem for
convex functions in $\mathcal{C}_1$ due to Donaldson. 

We also give an alternative description of the optimal destabiliser:
\begin{propL2lower}
  Consider the set $E\subset L^2(P)$ defined by
  \[ E = \{ h\in L^2\,|\, \mathcal{L}_h(f)\geqslant 0 \text{ for all
  convex } f\}.\]
  If $\Phi$ is the optimal destabilising convex function we found above,
  then $B = \hat{S}-\Phi$ is the unique minimiser of the $L^2$ norm for
  functions in $E$.  
\end{propL2lower}

The above two results show that 
\begin{equation}\label{eq:calinf}
  \inf_{h\in E} \Vert h-\hat{S}\Vert_{L^2} = \sup_{f\text{ convex}}
\frac{-\mathcal{L}(f)}{\Vert f\Vert_{L^2}}.
\end{equation}
In view of a conjecture of Donaldson's in~\cite{Don02} (see
Conjecture~\ref{conj:Don} in the next section), 
one can think of $E$ as the closure in $L^2$ of the possible scalar curvature
functions of torus invariant metrics on the toric variety. Thus
Equation (\ref{eq:calinf}) 
should be compared to another conjecture of Donaldson's (see~\cite{Don05})
saying that the
infimum of the Calabi functional is given by the supremum of the normalised
Futaki invariants over all test-configurations. Recall that the Calabi
functional is defined to be the $L^2$-norm of $S(\omega)-\hat{S}$ where
$S(\omega)$  is the scalar curvature of a K\"ahler metric $\omega$ and $\hat{S}$
is its average. In our toric setting this conjecture is
\begin{conj}\label{conj:calinf}
  For a polarised toric variety $(X,L)$ we have
  \[ \inf_{\omega\in c_1(L)} \Vert S(\omega)-\hat{S}\Vert_{L^2} =
  \sup_{f\text{ convex}} \frac{-\mathcal{L}(f)}{\Vert f\Vert_{L^2}},
  \]
  where $f$ runs over convex functions on the moment polytope $P$.
\end{conj}
\noindent Instead of trying to show that Conjecture~\ref{conj:Don}
implies this conjecture, we will show in Section~\ref{sec:Calabi} that
it holds if the Calabi flow exists for all time. 

In Section~\ref{sec:HN} 
we show that if the optimal convex function $\Phi$ that we found
above is piecewise linear, then we obtain a canonical decomposition of the
polytope into semistable pieces, ie. an analogue of the Harder-Narasimhan
filtration. The pieces are given by the maximal subpolytopes on which $\Phi$ is
linear. For the precise statement see Theorem~\ref{thm:piecewiseHN}. 
When $\Phi$ is not piecewise linear then in the same way it defines a
decomposition into infinitely many pieces. We discuss the conjectured
relationship between these decompositions and the Calabi flow. 

In the final Section~\ref{sec:Calabi}
we study the Calabi flow on a toric variety. This is a fourth order parabolic
flow in a fixed K\"ahler class defined by 
\[ \frac{\partial \phi_t}{\partial t} = S(\omega_t), \]
where $\omega_t=\omega+i\partial\overline{\partial}\phi_t$ is a path of K\"ahler
metrics and $S(\omega_t)$ is the scalar curvature. 
It was introduced by Calabi in~\cite{Cal82} in order to find extremal
K\"ahler metrics. It is known that the
flow exists for a short time (see Chen-He~\cite{CH06}), but the
long time existence has only been shown in special cases. For the case of
Riemann surfaces see Chru\'sciel~\cite{Chru91} (and also~\cite{Chen01}
and~\cite{Stru02}). For ruled manifolds, restricting to metrics of
cohomogeneity one see~\cite{Gua05}. For general K\"ahler manifolds
long time existence has
been shown in~\cite{CH06}, assuming that the Ricci curvature
remains bounded. 

Under the assumption that it exists for all time, 
we show that the Calabi flow minimises the Calabi
functional. More precisely we show 
\begin{thmcalflow}
  Suppose that $u_t$ is a solution of the Calabi flow for all
  $t\in[0,\infty)$. Then
  \[ \lim_{t\to\infty} \Vert S(u_t)-\hat{S}+\Phi \Vert_{L^2} 
  =0,\]
  where $\Phi$ is the optimal destabilising convex function from
  Theorem~\ref{thm:worst_destab}. Moreover 
  \[ \Vert\Phi\Vert_{L^2} = \inf_{u\in\mathcal{S}}\Vert S(u) -
  \hat{S}\Vert_{L^2}. \]
\end{thmcalflow}
\noindent Here the $u_t$ are symplectic potentials on the polytope defining torus
invariant metrics on the toric variety. It follows from this result
that existence of the
Calabi flow for all time
implies Conjecture~\ref{conj:calinf}.
The proof of the result relies on studying the behaviour of some
functionals introduced in~\cite{Don02} generalising the well known 
Mabuchi functional, and is similar to a previous
result by the author on ruled surfaces (see~\cite{GSz07_1}).

\subsection*{Acknowledgements}
I would like to thank Simon Donaldson, Dimitri Panov, Jacopo Stoppa and
Valentino Tosatti for helpful conversations.

\section{Preliminaries}\label{sec:prelim}

In this section we present some of the definitions and results following
Donaldson~\cite{Don02} that we will need in the paper.
We first describe how to write metrics on
a toric variety in terms of symplectic potentials (see
Guillemin~\cite{Gui94}).
Let $(X,L)$ be a polarised toric variety of dimension $n$. 
There is a dense free open orbit of $(\mathbf{C}^*)^n$ inside $X$ which
we denote by $X_0$. Let us choose complex coordinates $w_1,\ldots,
w_n\in\mathbf{C}^*$. On the covering space $\mathbf{C}^n$ we have
coordinates $z_i=\log w_i=\xi_i + \sqrt{-1}\eta_i$. A
$T^n=(S^1)^n$-invariant metric on $\mathbf{C}^n$ can be written as
$\omega = 2i\overline{\partial}{\partial}\phi$ where $\phi$ is a
function of $\xi_1,\ldots,\xi_n$. This means that
\[ \omega = \frac{\sqrt{-1}}{2}\sum_{i,j} \frac{\partial^2\phi}{
\partial\xi_i \partial\xi_j} dz_i\wedge d\overline{z}_j,\]
so we need $\phi$ to be strictly convex. 

The $T^n$ action on $\mathbf{C}^n$ is Hamiltonian with respect to
$\omega$ and has moment map
\[ m(z_1,\ldots,z_n) = \left( \frac{\partial\phi}{\partial\xi_i}\right).
\]
If $\omega$ compactifies to give a metric representing the first Chern class
$c_1(L)$ then the image of $m$ is an integral polytope
$P\subset\mathbf{R}^n$. The \emph{symplectic potential} of the metric is
defined to be the Legendre transform of $\phi$: for $\underline{x}\in P$
there is a unique point $\underline{\xi}=\underline{\xi}(\underline{x})
\in\mathbf{R}^n$ where $\frac{\partial\phi}{\partial\xi_i}=x_i$, and the
Legendre transform $u$ of $\phi$ is 
\begin{equation}\label{eq:sympot}
	u(\underline{x}) = \sum_i x_i\xi_i - \phi(\underline{\xi}).
\end{equation}
This is a strictly convex function and the metric in the coordinates $x_i,
\eta_i$ is given by
\begin{equation}\label{eq:metricuij}
  u_{ij} dx^i dx^j + u^{ij} d\eta^i d\eta^j,
\end{equation}
where $u^{ij}$ is the inverse of the Hessian matrix $u_{ij}$. 

It is important to study the behaviour of $u$ near the boundary of $P$.
Suppose that $P$ is defined by linear
inequalities $h_k(x) > c_k$, where each $h_k$ induces a primitive integral
function $\mathbf{Z}^n\to\mathbf{Z}$. 
Write $\delta_k(x)=h_k(x) - c_k$ and define the
function
\[ u_0(x) = \sum_k \delta_k(x)\log \delta_k(x), \]
which is a continuous function on $\overline{P}$, smooth in the
interior. It turns out that the boundary behaviour of $u_0$ models the
required boundary behaviour for a symplectic potential $u$ to give
a metric on $X$ in the class $c_1(L)$. More precisely let
$\mathcal{S}$ be the set of continuous, convex functions $u$ on
$\overline{P}$ such that $u-u_0$ is smooth on $\overline{P}$. Then 
(see Guillemin~\cite{Gui94}) 
there is a one-to-one correspondence
between $T$-invariant K\"ahler potentials $\psi$ on $X$, and symplectic
potentials $u$ in $\mathcal{S}$. 

The scalar curvature of the metric defined by $u\in\mathcal{S}$ was
computed by Abreu~\cite{Abr98}, and up to a factor of two is 
given by
\[ S(u) = -\frac{\partial^2 u^{ij}}{\partial x^i\partial x^j}, \]
where $u^{ij}$ is the inverse of the Hessian of $u$, and we sum over the
indices $i,j$.

Define the measure $d\mu$ on $P$ to be the $n$-dimensional Lebesgue
measure. 
Let us also define a measure $d\sigma$ on the boundary $\partial P$ as
follows. On the face of $P$ defined by $h_k(x) = c_k$, we
choose $d\sigma$ so that $d\sigma\wedge dh_k = \pm d\mu$. For example if
the face is parallel to a coordinate hyperplane, then the measure
$d\sigma$ on it is the
standard $n-1$-dimensional Lebesgue measure. Let us write $P^*$ for the
union of $P$ and its codimension one faces and write $\mathcal{C}_1$ for
the set of continuous convex functions on $P^*$ which are integrable on
$\partial P$. 
For a function $A\in L^2(P)$ let us define the functional
\[ \mathcal{L}_A(f) = \int_{\partial P} f\,d\sigma - \int_P Af\, d\mu,\]
defined for convex functions $f\in\mathcal{C}_1\cap L^2$. 
Let us recall the following integration by parts result from~\cite{Don02}
or~\cite{Don05_1}. 

\begin{lem}\label{lem:intpart}
	Let $u\in\mathcal{S}$ and $f$ a continuous convex function on
	$\overline{P}$, smooth in the interior. Then
	$u^{ij}f_{ij}$ is integrable on $P$ and
	\[ \int_P u^{ij} f_{ij}\, d\mu = \int_P (u^{ij})_{ij} f\, d\mu +
	\int_{\partial P} f\, d\sigma. \]
\end{lem}

\noindent It follows that if we let $A=S(u)$ for some $u\in\mathcal{S}$ then
\[ \mathcal{L}_A(f) = \int_P u^{ij} f_{ij}\,d\mu .\]
In particular $\mathcal{L}_A(f)\geqslant 0$ for all convex $f$ with equality
only if $f$ is affine linear. The converse is conjectured by Donaldson.
\begin{conj}[see \cite{Don02}]\label{conj:Don}
  Let $A$ be a smooth bounded function on $P$. If $\mathcal{L}_A(f) > 0$ for all
  non affine linear convex functions $f\in\mathcal{C}_1$ 
  then there exists a symplectic potential $u\in\mathcal{S}$
  with $S(u)=A$. 
\end{conj}
In the special case when $A=\hat{S}$ we simply write $\mathcal{L}$ instead of
$\mathcal{L}_A$. The condition $\mathcal{L}(f)\geqslant0$ for all convex $f$ is
called K-semistability. 
If in addition we require that equality only holds for
affine linear $f$ then it is called K-polystability. Technically we
should say ``with respect to toric test-configurations'', but since we
only deal with toric varieties we will neglect this. For more details on
stability,
in particular on how to construct a test-configuration given a rational
piecewise-linear convex function and how to compute the Futaki invariant,
see~\cite{Don02}.

\section{Optimal destabilising convex functions} \label{sec:worstdestab}

The aim of this section is to show that for an unstable toric variety 
there exists a ``worst destabilising test-configuration''. We introduce
the normalised Futaki invariant
\[ W(f) = \frac{\mathcal{L}(f)}{\Vert f\Vert_{L^2}},\]
for non-zero convex functions $f$ and let $W(0)=0$. The worst destabilising
test-configuration is a convex function minimising $W$. It will only
define a genuine test-configuration if it is rational and piecewise
linear, so in general we should think of it as a limit of
test-configurations.

\begin{thm}\label{thm:worst_destab} Let the toric variety with moment
  polytope $P$ be unstable. Then there exists a convex minimiser
  $\Phi\in\mathcal{C}_1\cap L^2(P)$ for
  $W$ which is unique up
  to scaling. Let us fix the scaling by requiring 
  that 
  \[ \mathcal{L}(\Phi)=-\Vert\Phi\Vert^2_{L^2}.\]
  Letting $B=\hat{S}-\Phi$, we then have $\mathcal{L}_B(f)\geqslant 0$
  for all convex functions $f$ and $\mathcal{L}_B(\Phi)=0$. Conversely
  these two conditions characterise $\Phi$. 
\end{thm}

\begin{proof}
	Let $A$ be the unique affine linear function so that
        $\mathcal{L}_A(f)=0$ for all affine linear $f$.
	We will show in Proposition
	\ref{prop:Wmin} the existence of a
	convex $\phi\in\mathcal{C}_1\cap L^2$ such that letting
	$B=A-\phi$ we have 
	\[\begin{aligned} \mathcal{L}_B(f) &\geqslant 0\quad\text{ for all
		convex } f\\
		\mathcal{L}_B(\phi) &= 0.
	\end{aligned} \]
	In addition $\phi$ is $L^2$-orthogonal to the affine linear
	functions.
	Let $\Phi = \phi + \hat{S}-A$. We show that this $\Phi$ satisfies the
	requirements of the theorem. 
	
	Note that $B=\hat{S}-\Phi$ with the same $B$ as above, and we
	also have $\mathcal{L}_B(\Phi)=0$. By definition we have
	\[ \mathcal{L}(f) = \mathcal{L}_B(f) + \langle B-\hat{S},f\rangle. \]
	In particular, for all convex $f$
	\[ \mathcal{L}(f) \geqslant \langle B-\hat{S}, f\rangle\geqslant -\Vert
	B-\hat{S}\Vert_{L^2}\Vert f\Vert_{L^2}, \]
	ie. $W(f)\geqslant -\Vert\Phi\Vert_{L^2}$. On the other hand $W(\Phi) =
	-\Vert\Phi\Vert_{L^2}$, so that $\Phi$ is indeed a minimiser for $W$. 

	To show uniqueness, suppose that there are two minimisers $\Phi_1$ and
	$\Phi_2$, and normalise them so that
	$\Vert\Phi_1\Vert_{L^2}=\Vert\Phi_2\Vert_{L^2}$, which in turn implies
	$\mathcal{L}(\Phi_1)=\mathcal{L}(\Phi_2)$. If $\Phi_1$ is not a scalar
	multiple of $\Phi_2$, then we have
	\[ \Vert\Phi_1+\Phi_2\Vert_{L^2} < 2\Vert\Phi_1\Vert_{L^2},\]
	so that
	\[ W(\Phi_1+\Phi_2) =
	\frac{2\mathcal{L}(\Phi_1)}{\Vert\Phi_1+\Phi_2\Vert_{L^2}} <
	\frac{\mathcal{L}(\Phi_1)}{\Vert\Phi_1\Vert_{L^2}},\]
	contradicting that $\Phi_1$ was a minimiser (note that
	$\mathcal{L}(\Phi_1) < 0$).
\end{proof}

\begin{prop}\label{prop:Wmin}
        There exists a convex function $\phi$ such that $B=A-\phi$ (where $A$
	is as in the previous proof) satisfies
	\[ \mathcal{L}_B(f) \geqslant 0\text{ for all convex } f \text{
	and } \mathcal{L}_B(\phi)=0.\] 
	In addition $\phi$ is $L^2$-orthogonal to the affine linear
	functions.
\end{prop}
The proof of this will take up most of this section. 
Suppose the origin is contained in the interior of $P$. We call a convex
function normalised if it is non-negative and vanishes at the origin. The key to
our proof is a compactness result for normalised convex functions given by
Donaldson in~\cite{Don02}. In order to apply it we need to reduce our
minimisation problem to
one where we can work with normalised convex functions. Let $A$ be the unique
affine linear function so that $\mathcal{L}_A(f)=0$ for all affine linear $f$ as
before, and let us introduce the functional 

\[ W_A(f) = \frac{\mathcal{L}_A(f)}{\Vert f\Vert_{L^2}}.\] 

  \begin{prop}\label{prop:extremalopt}
	  Suppose that $\mathcal{L}_A(f)<0$ for some convex $f$. Then there
	  exists a convex minimiser $\phi\in\mathcal{C}_1\cap L^2$ for $W_A$. 
  \end{prop}
  \begin{proof}
	  We introduce one more functional 
	  \[ \tilde{W}_A(f) = \frac{\mathcal{L}_A(f)}{\Vert f -
	  \pi(f)\Vert_{L^2}}, \]
	  where $\pi$ is the $L^2$-orthogonal projection onto affine linear
	  functions. We define $\tilde{W}_A(f)=0$ for affine linear $f$.
	  The advantage of $\tilde{W}_A$ is that it is invariant
	  under adding affine linear functions to $f$, so we can
	  restrict to looking at normalised convex functions. 
	  In addition if we find a
	  minimiser $g$ for $\tilde{W}_A$, then clearly $g-\pi(g)$ is a
	  minimiser for $W_A$.

	  The first task is to show that $\tilde{W}_A$ is bounded from below.
	  For this note that for a normalised convex function $f$ we have
	  \[ \mathcal{L}_A(f) \geqslant -\int_P Af\, d\mu \geqslant -\Vert
	  A\Vert_{L^2}\Vert f\Vert_{L^2}.\]
	  By Lemma~\ref{lem:normineq} this implies
	  \[ \mathcal{L}_A(f) \geqslant -C\Vert A\Vert_{L^2}
	  \Vert f-\pi(f)\Vert_{L^2},\]
	  so that $\tilde{W}_A(f)\geqslant -C\Vert A\Vert_{L^2}$.

	  Now we can choose a minimising sequence $f_k$ for $\tilde{W}_A$, where
	  each $f_k$ is a normalised convex function. In addition we can scale
	  each $f_k$ so that
	  \begin{equation}\label{eq:scaling}
	    \int_{\partial P} f_k\,d\sigma = 1.
	  \end{equation}
	
	  According to Proposition 5.2.6. in~\cite{Don02} we can choose a
	  subsequence which converges  uniformly over compact subsets of $P$ to
	  a convex function which has a continuous extension to a
	  function $\phi$
	  on $P^*$ with 
	  
	  \[ \int_{\partial P} \phi\,d\sigma \leqslant \lim\inf
	  \int_{\partial P} f_k\, d\sigma. \] 

	  As in~\cite{Don02} we find that this implies 
	  \begin{equation} \label{eq:futakilower}
	     \mathcal{L}_A(\phi)\leqslant\lim\inf\mathcal{L}_A(f_k).  
          \end{equation}

	  If we can show that at the same time
	  \begin{equation} \label{eq:L2upper}
		  \Vert \phi-\pi(\phi)\Vert_{L^2}\leqslant \lim\inf\Vert f_k-\pi(f_k)
		  \Vert_{L^2}
	  \end{equation}
	  then together with the previous inequality this will imply
	  that $\phi$ 
	  is a minimiser of $\tilde{W}_A$ and also $\phi\in L^2$. 

	  In order to show Inequality~\ref{eq:L2upper} we first show that the
	  $f_k-\pi(f_k)$ are uniformly bounded in $L^2$. To see this, note that 
	  \[ |\mathcal{L}_A(f_k)| \leqslant \int_{\partial P} f_k\, d\sigma +
	  \Vert A\Vert_{L^\infty}\int_P f_k\,d\mu \leqslant C\int_{\partial P}
	  f_k\,d\sigma = C,\]
	  for some $C>0$ depending on $A$, since the boundary integral of a
	  normalised convex function controls the integral on $P$. Since $f_k$
	  is a minimising sequence for $\tilde{W}_A$, this implies that for some
	  constant $C_1$ we have 
	  \[\Vert f_k - \pi(f_k)\Vert_{L^2}\leqslant C_1.\]
	  Now from the fact that $f_k\to \phi$ uniformly on
	  compact sets $K\subset\subset P$ we have 
	  \[ \Vert \phi-\pi(\phi)\Vert_{L^2(K)} = \lim_k \Vert f_k - \pi(f_k)
	  \Vert_{L^2(K)}\leqslant
	  \lim\inf_k \Vert f_k-\pi(f_k)\Vert_{L^2(P)},\]
	  and taking the limit over compact subsets $K$, we get the
	  Inequality~\ref{eq:L2upper}.  
  \end{proof}

  We now prove a lemma that we have used in this proof.

\begin{lem}\label{lem:normineq}
	There is a constant $C>0$ such that for all normalised
  convex functions $f$ we have
  \[ \Vert f\Vert_{L^2} \leqslant C\Vert f - \pi(f)\Vert_{L^2}.\]
\end{lem}
\begin{proof}
  We will prove that for some $\epsilon > 0$ we have 
  \begin{equation}\label{ineq:proj}
    \Vert \pi(f)\Vert_{L^2} \leqslant (1-\epsilon)\Vert f\Vert_{L^2}.
  \end{equation}
  The result follows from this, with $C=\epsilon^{-1}$. 

  Suppose Inequality~\ref{ineq:proj} does not hold so 
  that there is a sequence of normalised convex functions $f_k$
  such that $\Vert f_k\Vert_{L^2}=1$ and $\Vert \pi(f_k)\Vert_{L^2}\to
  1$. By possibly taking a subsequence we can assume that $f_k$
  converges weakly to $f$. The projection $\pi$ onto a finite
  dimensional space is compact, so $\pi(f_k)\to \pi(f)$ in norm. In
  particular $\Vert\pi(f)\Vert_{L^2}=1$. It follows that $\Vert
  f\Vert_{L^2}=1$ since the norm is lower semicontinuous. Hence
  $f=\pi(f)$ ie. $f$ is affine linear and also the convergence $f_k\to f$
  is strong. Then there is a subsequence which we also denote by $f_k$
  which converges pointwise almost everywhere to $f$. Since the $f_k$
  are normalised convex functions it is easy to see
  that $f$ must be zero, which is a contradiction, so
  Inequality~\ref{ineq:proj} holds. 
\end{proof}

Finally we can prove Proposition~\ref{prop:Wmin}, which then completes
the proof of Theorem~\ref{thm:worst_destab}.

\begin{proof}[Proof of Proposition~\ref{prop:Wmin}]
  	If $\mathcal{L}_A(f)\geqslant 0$ for all convex $f$ then we take
	$\phi=0$.
	Otherwise Proposition~\ref{prop:extremalopt} implies that there
	is a minimiser $\phi$ for $W_A$, and by rescaling $\phi$ we can ensure
	that
	\[ \mathcal{L}_A(\phi) = -\Vert \phi\Vert^2_{L^2}.\]

	\noindent
	Note that $\phi$ is $L^2$-orthogonal to the affine linear functions
	because it minimises $W_A$. By definition we have that for all $f$
	\[ \mathcal{L}_B(f) = \mathcal{L}_A(f) + \langle A-B,f\rangle_{L^2} =
	\mathcal{L}_A(f) +\langle \phi,f\rangle_{L^2}.\]
	It follows that 
	\[ \mathcal{L}_B(\phi) = \mathcal{L}_A(\phi) + \Vert \phi
	\Vert^2_{L^2} = 0. \]
	
	Now consider perturbations of the form $\phi_t=\phi + t\psi$ 
	which are convex for
	sufficiently small $t$, $\langle \phi,\psi\rangle_{L^2}=0$, 
	but $\psi$ is not necessarily convex. Since $\phi$ 
	minimises $W_A$, we must have
	\[ \left.\frac{d}{dt}\right|_{t=0} \mathcal{L}_A(\phi_t) \geqslant 0, \]
	ie. $\mathcal{L}_A(\psi)\geqslant 0$. 

	We can write any convex function $f$ as $f = c\cdot \phi + \psi$, where
	$c\in\mathbf{R}$ and $\langle \phi,\psi\rangle_{L^2}=0$. Since $\phi$ 
	is convex, we
	have that for all $K>\max\{-c,0\}$ the function
	\[ \frac{f+K\phi}{c+K} = \phi + \frac{1}{c+K}\psi \]
	is convex, so by the previous argument we must have
	$\mathcal{L}_A(\psi)\geqslant 0$. This means that
	\[ \mathcal{L}_B(f) = c\cdot\mathcal{L}_B(\phi) +
	\mathcal{L}_A(\psi) +
	\langle \phi,\psi\rangle = \mathcal{L}_A(\psi)\geqslant 0. \]
	This is what we wanted to show. 
\end{proof}

We finally give a slightly different variational characterisation of
$\Phi$.

\begin{prop}\label{prop:L2lower}
  Consider the set $E\subset L^2(P)$ defined by
  \[ E = \{ h\in L^2\,|\, \mathcal{L}_h(f)\geqslant 0 \text{ for all
  convex } f\}.\]
  If $\Phi$ is the optimal destabilising convex function we found above,
  then $B = \hat{S}-\Phi$ is the unique minimiser of the $L^2$ norm for
  functions in $E$.  
\end{prop}
\begin{proof}
  Suppose that $h\in E$. Since $\Phi$ is convex we have
  \begin{equation}\label{eq:Lh}
    0\leqslant \mathcal{L}_h(\Phi) = \mathcal{L}(\Phi) + \langle \hat{S}-h,
    \Phi\rangle.
  \end{equation}
  Since we have 
  $\mathcal{L}(\Phi)=-\Vert\Phi\Vert_{L^2}^2$, we get
  \begin{equation}\label{eq:CS}
    \Vert\Phi\Vert_{L^2}^2 \leqslant \langle \hat{S}-h, \Phi\rangle \leqslant
    \Vert \hat{S}-h\Vert_{L^2} \Vert\Phi\Vert_{L^2},
  \end{equation}
  ie.  
  \[\Vert\Phi\Vert_{L^2}\leqslant\Vert \hat{S}-h\Vert_{L^2}.\]
  Since
  $\mathcal{L}_h(1)=0$ if follows from (\ref{eq:Lh}) that 
  $\hat{S}-h$ is orthogonal to constants. So is $\Phi$, therefore 
  the previous inequality implies
  \[ \Vert B\Vert_{L^2} = \Vert \hat{S}-\Phi\Vert_{L^2}\leqslant \Vert
  h\Vert_{L^2}.\]

  Equality in (\ref{eq:CS}) 
  can only occur if $\hat{S}-h$ is a positive scalar multiple of
  $\Phi$, but then it must be equal to $\Phi$ by (\ref{eq:Lh}). 
\end{proof}

  Note that we can rewrite the definition of the set $E$ as saying that $h\in E$
  if and only if for all convex $f\in \mathcal{C}_1\cap L^2$ we have
  \[ \langle h, f\rangle \leqslant \int_{\partial P} f\, d\sigma.\]
  Thus $E$ is the intersection of a collection of closed affine half
  spaces, and is therefore a closed convex set in $L^2$. It follows
  that there exists a unique minimiser for the $L^2$-norm in $E$. From
  this point of view the content of Theorem~\ref{thm:worst_destab} is
  that this minimiser is concave.

  Also note that Theorem~\ref{thm:worst_destab} still holds
  when we use a different boundary measure $d\sigma$ in defining the
  functional $\mathcal{L}$. In particular when $d\sigma$ is zero on some
  faces, which is a situation we encounter in the next section. The
  proof is identical, except in the normalisation (\ref{eq:scaling}) 
  we still use the old $d\sigma$.

\section{Harder-Narasimhan filtration}\label{sec:HN}

In this section we would like to study the problem of decomposing an unstable
toric variety into semistable pieces. 
This is analogous to the Harder-Narasimhan filtration of an
unstable vector bundle. After making the problem more precise, we will show that
we obtain such a decomposition when the 
optimal destabilising convex function found
in Section~\ref{sec:worstdestab} is piecewise linear. After that we
discuss the implications of such a decomposition and we also look at the
case when the optimal destabiliser is not piecewise linear. 
For convenience we introduce the following terminology. 

\begin{defn}\label{def:semistab}
        Let $Q\subset\mathbf{R}^n$ be a polytope, and let $d\sigma$ be a 
	measure on the
	boundary $\partial Q$. It may well be zero on some edges. Let $A$ be the
	unique affine linear function on $Q$ such that $\mathcal{L}_A(f)=0$ for
	all affine linear functions $f$, where 
	\[ \mathcal{L}_A(f) = \int_{\partial Q} f\, d\sigma - \int_Q
	Af\, d\mu\]
	as before, with $d\mu$ being the standard Lebesgue measure (but
	$d\sigma$ can be different from the one we used before). 

	We say that $(Q,d\sigma)$ is \emph{semistable}, if
	$\mathcal{L}_A(f)\geqslant 0$ for all convex functions. It is
	\emph{stable} if in addition $\mathcal{L}_A(f)=0$ only for affine linear
	$f$. 

	Let us say that a concave $B\in L^2$ 
	is the \emph{optimal density function} for $(Q,d\sigma)$ if
	$\mathcal{L}_B(f)\geqslant 0$ for all convex $f$, and
	$\mathcal{L}_B(B)=0$. Note that such a $B$ exists and is unique
	by the results in Section~\ref{sec:worstdestab}. 
\end{defn}

\begin{rem}
  \begin{enumerate}
    \item
	If in the above definition $Q$ is the moment polytope of a toric variety
	and $d\sigma$ is the canonical boundary measure we have defined before
	then $(Q,d\sigma)$ is stable if and only if the toric variety
	is relatively K-stable (see~\cite{GSzThesis}). It is conjectured
	that in this case the toric variety admits an
	extremal metric (see~\cite{Don02}). 
      \item
	If the measure $d\sigma$ is the canonical measure on \emph{some} edges
	but zero on some others corresponding to a divisor $D$, then it is
	conjectured (see~\cite{Don02}) that stability of $(Q,d\sigma)$ implies
	that the toric variety admits a complete extremal metric on the
	complement of $D$. 
      \item
	Also note that $(Q,d\sigma)$ is semistable precisely when its
	optimal density function is affine linear.  
    \end{enumerate}
\end{rem}

With this terminology we can state precisely what we would like to show (see
also Donaldson~\cite{Don02}).

\begin{conj}\label{conj:HN}
  	Let $(P,d\sigma)$ be the moment polytope of a polarised toric variety
	with the canonical boundary measure $d\sigma$. If $(P,d\sigma)$ is not
	semistable, then it has a subdivision into finitely many
	polytopes $Q_i$, such that if $d\sigma_i$ is the restriction of $d\sigma$
	to the faces of $Q_i$, then each $(Q_i,d\sigma_i)$ is semistable. 
\end{conj}

Our main tool is the theorem of Cartier-Fell-Meyer~\cite{CFM} about
measure majorisation. We state it in a slightly different form from the
original one.

\begin{thm}[Cartier-Fell-Meyer]\label{thm:CFM}
  Suppose $d\lambda$ is a signed measure supported on the closed convex
  set $P$. Then 
  \begin{equation}\label{eq:majorise}
    \int_P f\,d\lambda \geqslant 0 
  \end{equation}
  for all convex functions $f$ if and only if $d\lambda$ can be
  decomposed as
  \[ d\lambda = \int_P (T_x - \delta_x)\, d\nu(x),\]
  where each $T_x$ is a probability measure with barycentre $x$, the
  measure $\delta_x$ is the point mass at $x$ and $d\nu(x)$ is a
  non-negative measure on $P$.
\end{thm}

Note that
the converse of the theorem follows easily from Jensen's inequality:
\begin{lem}[Jensen's inequality] Let $T_x$ be a probability measure with
	barycentre $x$. Then for all convex functions $f$ we have
	\[ f(x) \leqslant \int f(y)\, dT_x(y).\]
	Equality holds if and only if $f$ is affine linear on the convex hull of
	the support of $T_x$.
\end{lem}

Our result is the following
\begin{thm}\label{thm:piecewiseHN}
	Suppose $(P,d\sigma)$ is not semistable, and let $\Phi$ be the optimal
	destabilising convex function found in Section~\ref{sec:worstdestab}. If
	$\Phi$ is piecewise linear, then the maximal 
	subpolytopes of $P$ on which $\Phi$ is linear give the decomposition
	of $P$ into semistable pieces required by
	Conjecture~\ref{conj:HN}.
\end{thm}

\begin{proof}
  Let $\Phi$ be the optimal destabilising convex function, and assume
  that it is piecewise linear. Let us write $(Q_i, d\sigma_i)$ for the
  maximal subpolytopes of $P$ on which $\Phi$ is linear, with
  $d\sigma_i$ being the restriction of $d\sigma$ to the boundary of
  $Q_i$. According to Theorem~\ref{thm:worst_destab} we have 
  \[ \mathcal{L}_B(f) \geqslant 0 \]
  for all convex $f$, where $B=\hat{S}-\Phi$. This means that the signed
  measure $d\sigma - B\,d\mu$ satisfies (\ref{eq:majorise}). It follows
  that there is a decomposition
  \[ d\sigma - B\,d\mu = \int_P (T_x - \delta_x)\, d\nu(x).\]
  Since in addition $\mathcal{L}_B(\Phi)=0$, we have that for  almost
  every $x$ with respect to $d\nu$, the restriction of $\Phi$ to the convex hull of the
  support of $T_x$ is linear. This means that for almost every
  $x$ (w.r.t. $d\nu$) the support of $T_x$ is contained in some $Q_i$, 
  so that for each $i$ we have
  \[ d\sigma_i - B\,d\mu|_{Q_i} = \int_{Q_i} (T_x-\delta_x)\, d\nu(x).\]
  The Jensen inequality implies that for every convex function $f$ on
  $Q_i$ we have
  \[ \int_{\partial Q_i} f\,d\sigma - \int_{Q_i} Bf\,d\mu \geqslant 0.\]
  Since $B$ is linear when restricted to $Q_i$ this means that
  $(Q_i,d\sigma_i)$ is semistable.
\end{proof}

\begin{rem} Note that by the uniqueness of the optimal density function
  we get a canonical decomposition into semistable
  pieces $Q_i$ if we require that the affine linear densities
  corresponding to the $Q_i$ fit together to form a concave function on
  $P$. This corresponds to the condition that in the Harder-Narasimhan
  filtration of an unstable vector bundle
  the slope of the successive quotients is decreasing.
\end{rem}

Suppose as in the theorem that $\Phi$ is piecewise linear and that in
addition all the pieces $Q_i$ that we obtain are in fact stable 
(not just semistable). Then
conjecturally they admit complete extremal metrics. We think of this
purely in terms of symplectic potentials on polytopes, and not in terms
of the complex geometry because when the pieces are not \emph{rational}
polytopes then they do not correspond to complex varieties. So an
extremal metric on a piece $Q$ is a strictly smooth convex function $u$
on $Q$ which has the same asymptotics as a symplectic potential near
faces of $Q$ that lie on $\partial P$, but which has the asymptotics
$-a\log d$ near interior faces. Here $a>0$ is a function on the face 
and $d$ is the
distance to the face. Piecing together these functions we obtain a
``symplectic potential'' $u$ on $P$, which is singular along the
interior boundaries of the pieces $Q_i$, ie. along the codimension one locus
where $\Phi$ is not smooth. Conjecturally the Calabi flow should
converge to this singular symplectic potential. More precisely if $u_t$
is a solution to the Calabi flow, then the sequence of functions $u_t -
tB$ should converge to $u$ up to addition of an affine
linear function, where $B=\hat{S}-\Phi$ as usual.  
A decisive step in this direction would be to
show that along the flow the scalar curvature converges uniformly to
$B$. In the next section we show the much weaker result
that this is true in $L^2$ assuming that the flow exists for all time.

Suppose now that some of the pieces we obtain are semistable. In some
cases it may be possible to decompose these into a finite number of
stable pieces, to which the previous discussion applies. There may be
some semistable pieces though which do not have a decomposition into
finitely many stable pieces. For example suppose that $Q$ is a
trapezium, and that the measure $d\sigma$ is only non-zero on the two
parallel edges. Let us suppose for simplicity that $Q$ is the trapezium
in $\mathbf{R}^2$ with vertices $(0,0),(1,0),(1,l),(0,1)$ for some
$l>0$ and that $d\sigma$ is the Lebesgue measure on the vertical edges. 

\begin{prop} The trapezium $(Q,d\sigma)$ is semistable in the sense of
  Definition~\ref{def:semistab}. Moreover $\mathcal{L}_A(f)=0$ for all
  simple piecewise linear $f$ with crease joining the points
  $(0,u),(1,ul)$ for $0 < u < 1$. 
\end{prop}
Recall that a simple piecewise linear function is $\max\{h,0\}$ where
$h$ is affine linear. The line $h=0$ is called the crease.
\begin{proof}
  The first task is to compute the linear function $A$. This can be done
  easily by writing $A(x,y)=ax+by+c$ and solving the linear system of
  equations $\mathcal{L}_A(1),\mathcal{L}_A(x),\mathcal{L}_A(y)=0$ for
  $a,b,c$. As a
  result we obtain
  \[ A(x,y) = \frac{1}{l^2+4l+1}\Big[12(l^2-1)x - 6(l^2-2l-1)\Big]. \]
  It follows that
  \[\begin{aligned}
    \int_Q Af\,d\mu &= \int_0^1\int_0^{1+(l-1)x}
  Af\,dy\,dx \\
  &= \int_0^1\int_0^1 [1+(l-1)x]\, A(x)\, f\big(x,(1+(l-1)x)y'\big)
  \,dx\,dy',
  \end{aligned} \]
  where we have made the substitution $y'=y/(1+(l-1)x)$. Since for a
  fixed $y'$ the function $f\big(x, (1+(l-1)x)y'\big)$
  is convex in $x$, the following lemma tells us that 
  \[ \int_0^1 [1+(l-1)x]\,A(x)\,f\big(x,(1+(l-1)x)y'\big)\, dx \leqslant
  f(0,y') + l\cdot f(1,ly').\]
  Integrating over $y'$ as well get
  \[ \mathcal{L}_A(f) = \int_{\partial Q} f\,d\sigma - \int_Q Af\,d\mu
  \geqslant 0,\]
  which shows that $(Q,d\sigma)$ is semistable. 
  It is clear from the proof that if
  $f$ is linear when restricted to the line segments $y=u+u(l-1)x$ for
  $0 < u <1$ then $\mathcal{L}_A(f)=0$, which gives the second statement
  in the proposition.
\end{proof}

\begin{lem} Let $g:[0,1]\to\mathbf{R}$ be convex. Then we have
  \begin{equation}\label{eq:1d}
    \int_0^1 [1+(l-1)x]\,A(x)\,g(x)\,dx\leqslant g(0) + l\cdot g(1),
  \end{equation}
  where $A(x)$ is as in the previous proposition. Moreover equality
  holds only if $g$ is affine linear. 
\end{lem}
\begin{proof}
  By an approximation argument we can assume that $g$ is smooth.
  It can be checked directly that when $g$ is affine linear, we have
  equality in (\ref{eq:1d}), so we can also assume that $g(0)=0$ and
  $g'(0)=0$. We can then write
  \[ g(x) = \int_0^x g''(t)\cdot(x-t)\,dt = \int_0^1 g''(t)\cdot
  \max\{0,x-t\}\,dt.\]
  It follows that it is enough to check (\ref{eq:1d}) for the functions
  $g(x)=\max\{0,x-t\}$ for $0\leqslant t\leqslant 1$.  In other words we
  need to show that
  \[ \int_t^1 [1+(l-1)x]\,A(x)\,(x-t)\,dx - l(1-t)\leqslant 0,\]
  for $0\leqslant t\leqslant 1$. This expression is a quartic in $t$,
  whose roots include $t=0$ and $t=1$. It is then easy to see by
  explicit computation that the inequality holds, and equality only
  holds for $t=0,1$. This means that in (\ref{eq:1d}) equality can only
  hold if $g''(t)=0$ for almost every $t\in(0,1)$, ie. if $g$ is affine
  linear.
\end{proof}

As a consequence of the proposition we see that if we decompose the
measure $d\sigma - Ad\mu$ according to Theorem~\ref{thm:CFM} then for
almost every $x$
the $T_x$ that we obtain has support contained in one of the line
segments joining $(0,u),(1,ul)$ for some $0<u<1$. 
It is then clear that $(Q,d\sigma)$ does not have a
decomposition into finitely many stable pieces. On such semistable
pieces the Calabi flow is expected to collapse an $S^1$ fibration. 
This was predicted in~\cite{Don02} for the case when $Q$ is a 
parallelogram. Note that parallelograms correspond to product
fibrations whereas other rational trapeziums correspond to non-trivial
$S^1$ fibrations.

Finally let us see what we can say when $\Phi$ is not piecewise linear.
We can still decompose $P$ into the maximal subsets $Q_i$ on which
$\Phi$ is linear, but now we get infinitely many such pieces and many
will have dimension lower than that of $P$. We still have a
decomposition 
\[ d\sigma - B\,d\mu = \int_P (T_x - \delta_x)\,d\nu,\]
as in the proof of the theorem, but if $Q$ is a lower
dimensional piece, then we cannot simply restrict the measures $d\sigma$
and $B\,d\mu$ to $\partial Q$ and $Q$ respectively. This is similar to
the case of trapeziums above where the $Q_i$ are the line
segments joining the points $(0,u),(1,ul)$. The correct measure on the
line segment is given by $[1+(l-1)x]A(x)\,d\mu$ and on the boundary it's
a weighted sum of the values at the endpoints. 
The lemma shows that with respect to these
measures the line segments are stable. This is what we try to imitate in
the general case. 

Suppose then that $Q$ is such a lower dimensional piece and that we can find a 
closed convex neighbourhood $K$ of $Q$ with non-empty interior
such that $K\cap\partial P$ also has nonempty interior, and
for almost every
$x\in K$ the support of $T_x$ is contained in $K$.  
For each such $K$ we have
\[ \int_{\partial K} f\,d\sigma - \int_{K} Bf\,d\mu
\geqslant 0,\]
for all convex $f$. Suppose we have a sequence of such
neighbourhoods $K_i$ such that $\bigcap_i K_i=Q$.
Then, after perhaps choosing a subsequence of the $K_i$,
we can define a measure $d\tilde{\sigma}$ on $\partial Q$ by
\[ \int_{\partial Q} f\,d\tilde{\sigma} = \lim_i
\frac{1}{Vol( K_i,d\mu)}\int_{\partial
 K_i} \tilde{f}\,d\sigma,\]
where $\tilde f$ is a continuous extension of a continuous function $f$ on
$Q$. By choosing a further subsequence we can similarly
define $\tilde{B}\,d\mu$ and we
have that for every convex function $f$ on $Q$,
\[ \int_{\partial Q} f\,d\tilde{\sigma} - \int_{Q} f\tilde{B}\,d\mu
\geqslant 0,\]
since the corresponding inequality holds for each $ K_i$.
Note however that $\tilde{B}$ is not necessarily linear on $Q$, and
also $d\tilde{\sigma}$ is not necessarily a constant multiple of the
Lebesgue measure on the faces of $Q$. We thus obtain a decomposition of
$P$ into infinitely many pieces which are semistable in a suitable
sense. As in the case
of semistable trapeziums we discussed above, one expects
collapsing to occur along the Calabi flow. See the end of the next
section for an indication of why such collapsing must occur. 

We have not said how to construct a suitable sequence of closed
neighbourhoods $ K_i$. One way is to look at the
subdifferential of $\Phi$. At a point $x$ we write
$D\Phi(x)\subset(\mathbf{R}^n)^*$ for the closed set of supporting hyperplanes to
$\Phi$ at $x$. Choose $x_0$ in the interior of
$Q$, ie. in $Q\setminus\partial Q$. 
Note that for all interior points $D\Phi(x_0)$ is the same set, and for
points on the boundary of $Q$ it is strictly larger since $Q$ is a
maximal subset on which $\Phi$ is linear. Now we can simply define
\[  K_i = \{ x\in P\,|\, D\Phi(x)\cap \overline{B}_{1/i}(D\Phi(x_0))\not=0\},\]
where $\overline{B}_{1/i}(D\Phi(x_0))$ denotes the points of distance at most $1/i$
from $D\Phi(x_0)$. So $ K_i$ is the set of points with
supporting hyperplanes sufficiently close to those at $x_0$. These are necessarily
closed sets with nonempty interior (here we use that $Q$ is of strictly
lower dimension than $P$, so we can choose a sequence of points not in
$Q$ approaching an interior point of $Q$) and the intersection of all of them
is $Q$. Also note that for almost every $x$, any $y$ in the support of 
$T_x$ satisfies $D\Phi(x)\subset D\Phi(y)$ since $\Phi$ is linear on the
convex hull of $\text{supp}(T_x)$. This means that if $x\in K_i$ then
also $y\in K_i$. 

\section{The Calabi flow}\label{sec:Calabi}

In this section we study the Calabi flow on toric varieties, assuming
that it exists for all time. 
In terms of symplectic potentials the Calabi flow is given by 
the equation
\[ \frac{\partial}{\partial t} u_t = -S(u_t) = (u_t^{ij})_{ij},\]
where $u_t\in\mathcal{S}$ for $t\in[0,\infty)$. This can be seen by
differentiating the expression (\ref{eq:sympot}) defining the symplectic
potential and using the definition of the Calabi flow.

The aim of this section is to prove the following.
\begin{thm}\label{thm:calflow}
  Suppose that $u_t$ is a solution of the Calabi flow for all
  $t\in[0,\infty)$. Then
  \[ \lim_{t\to\infty} \Vert S(u_t)-\hat{S} +\Phi \Vert_{L^2} 
  =0,\]
  where $\Phi$ is the optimal destabilising convex function from
  Theorem~\ref{thm:worst_destab}. Moreover 
  \[ \Vert\Phi\Vert_{L^2} = \inf_{u\in\mathcal{S}}\Vert S(u) -
  \hat{S}\Vert_{L^2}. \]
\end{thm}

The first thing to note is that the Calabi functional is decreased under the
flow, ie. $\Vert S(u_t)\Vert_{L^2}$ is monotonically decreasing.
This is well-known and can be seen easily by computing the derivative. 

Recall that for $A\in L^\infty(P)$ we have defined the functional
\[\mathcal{L}_A(u) = \int_{\partial P} u\, d\sigma - \int_P Au\, d\mu. \]
Following \cite{Don02} let us also define 
\[\mathcal{F}_A(u) = -\int_P \log \det(u_{ij}) + \mathcal{L}_A(u),\]
for $u\in \mathcal{S}$. That this is well defined for all 
$u\in\mathcal{S}$ is shown in~\cite{Don02}. 
In the special case when $A=\hat{S}$, the functional 
$\mathcal{F}_{\hat{S}}$ is the same as the well known
Mabuchi functional and is also monotonically decreasing under the flow.
For general $A$ it is not monotonic, but will nevertheless be useful.

Finally recall that by Lemma~\ref{lem:intpart}, for $u,v\in\mathcal{S}$
we have
  \begin{equation}\label{eq:LSu}
    \mathcal{L}_{S(v)}(u) = \int_P v^{ij} u_{ij}\,d\mu. 
  \end{equation}

The proof of Theorem~\ref{thm:calflow} relies on the following two lemmas. 
\begin{lem} \label{lem:LAdecay}
  Choose some $v\in\mathcal{S}$. 
  If $u_t$ is a solution of the Calabi flow, we have
  \[ \mathcal{L}_{S(v)}(u_t)\leqslant C(1+t), \]
  for some constant $C>0$. 
\end{lem}
\begin{proof}
Write $A=S(v)$. Along the flow we have 
\[ \begin{aligned}
	\frac{d}{dt}\mathcal{F}_A(u_t) &= \int_P u_t^{ij} S(u_t)_{ij}\, d\mu -
	\mathcal{L}_A(S(u_t))\\
	&= \int_P (u_t^{ij})_{ij} S(u_t)\, d\mu + \int_P A S(u_t)\, d\mu \\
	&= \int_P (A-S(u_t))S(u_t)\, d\mu \leqslant C,
\end{aligned} \]
because the Calabi flow decreases the $L^2$-norm of $S(u_t)$. 
This implies that
\begin{equation}\label{eq:F_ineq}
  \mathcal{F}_A(u_t) \leqslant C(1+t)
\end{equation}
for some constant $C$. 

Now we use that $A=-(v^{ij})_{ij}$. We can write
\[\begin{aligned}
	\mathcal{F}_A(u) &= -\int_P \log\det(v^{ik}u_{kj})\, d\mu +
	\mathcal{L}_A(u) + C_1\\
	&= -\int_P \log\det(v^{ik}u_{kj})\, d\mu + \int_P v^{ij}u_{ij}\,
	d\mu +C_1,
\end{aligned} \]
for some constant $C_1$.
For a positive definite symmetric matrix $M$ we have $\log\det(M)
\leqslant \frac{1}{2} \mathrm{Tr}(M)$, applying the inequality $\log x < x/2$ 
to each eigenvalue. This implies that 
\[ \mathcal{F}_A(u) \geqslant \frac{1}{2}\mathcal{L}_A(u) + C_1. \]
Together with (\ref{eq:F_ineq}) this implies the result.
\end{proof}

\begin{lem}
  Fix some $v\in\mathcal{S}$, and write $A=S(v)$. For any $u\in\mathcal{S}$
  we have
  \[ -\int_P \log\det( u_{ij})\, d\mu \geqslant -C_1\log\mathcal{L}_A(u)
  -C_2,\]
  for some constants $C_1,C_2>0$.
\end{lem}
\begin{proof}
  Observe that
  \[ -\int_P \log\det(u_{ij}) = -\int_P\log\det (v^{ik}u_{kj})\,
  d\mu + C \]
The convexity of $-\log$ implies
\[ -\log\det (v^{ik}u_{kj}) \geqslant -C_1 \log\mathrm{Tr}(v^{ik}u_{kj}) - C_2 =
-C_1 \log v^{ij}u_{ij} - C_2.
\]
Therefore using the convexity of $-\log$ again,
\[ \begin{aligned}
  -\int_P \log\det(u_{ij})\,d\mu
  &\geqslant -C_1 \int_P \log v^{ij}u_{ij}\, d\mu -
	C_2 \\
	&\geqslant -C'_1 \log \int_P v^{ij}u_{ij}\, d\mu -C'_2 \\
	&= -C'_1 \log\mathcal{L}_A(u) -C'_2. 
\end{aligned} \]
\end{proof}

We are now ready to prove our theorem.
\begin{proof}[Proof of Theorem~\ref{thm:calflow}]
  Let us write $B = \hat{S} - \Phi$ as usual.
  Recall that $B$ satisfies $\mathcal{L}_{B}(f)\geqslant 0$ for
  all convex functions $f$, so that 
  \[ \mathcal{F}_B(u_t) \geqslant -\int_P \log\det(u_{t,ij})\,d\mu.\]
  The previous two Lemmas combined imply that 
  \[ \mathcal{F}_B(u_t) \geqslant -C_1\log(1+t) - C_2. \]
  At the same time we have
  \begin{equation}\label{eq:ddtF}
    \begin{aligned}
	\frac{d}{dt} \mathcal{F}_B(u_t) &= -\int_P (B-S(u_t))^2\, d\mu + \int_P
	B^2\, d\mu - \int_P BS(u_t)\,d\mu \\
	&= -\int_P (B-S(u_t))^2\, d\mu + \int_P u_t^{ij} B_{ij}\, d\mu -
	\mathcal{L}_B(B)\\
	&\leqslant -\int_P (B-S(u_t))^2\,d\mu
    \end{aligned}
  \end{equation}
  since $B$ is concave and $\mathcal{L}_B(B)=0$. 
  Together these inequalities imply that along some subsequence $u_k$ we
  have
  \[ \Vert S(u_k)-B\Vert_{L^2} \to 0. \]
  Since $\Vert S(u_t)\Vert_{L^2}$ is monotonically decreasing under the flow,
  this implies that 
  \[ \Vert S(u_t)\Vert_{L^2}\to \Vert B\Vert_{L^2}. \]
  In order to show that $S(u_t)\to B$ in $L^2$ not just along a subsequence,
  note that for $u\in\mathcal{S}$ we have
  \[ \mathcal{L}_{S(u)} (f) = \int_P u^{ij} f_{ij}\, d\mu \geqslant 0 \]
  for all continuous convex $f$, so that $S(u)$ is in the set $E$
  defined in Proposition~\ref{prop:L2lower}. Since $E$ is convex, we have that
  \[\frac{1}{2}(S(u_t) + B) \in E,\]
  so since $B$ minimises the $L^2$-norm in $E$, we have (suppressing the $L^2$
  from the notation) 
  \[\Vert S(u_t)+B\Vert  \geqslant 2\Vert B\Vert.\]
  It follows that 
  \[ \begin{aligned} \Vert S(u_t)-B\Vert^2 &= 2(\Vert S(u_t)\Vert^2 +
	  \Vert B\Vert^2) - \Vert S(u_t)+B\Vert^2 \\
	  &\leqslant 2(\Vert S(u_t)\Vert^2 +
	  \Vert B\Vert^2) - 4\Vert B\Vert^2 \\
	  &= 2(\Vert
	  S(u_t)\Vert^2-\Vert B\Vert^2)\to 0. \\
   \end{aligned} \]
   This proves the first part of the theorem.

   For the second part simply note that for $u\in\mathcal{S}$ we have $S(u)\in
   E$ as above, so that Proposition~\ref{prop:L2lower} implies that 
   \[ \Vert S(u)\Vert_{L^2} \geqslant \Vert \hat{S}-\Phi\Vert_{L^2}.\]
   Hence by the previous argument $\Vert\hat{S}-\Phi\Vert$ is in fact the infimum
   of $\Vert S(u)\Vert$ over $u\in\mathcal{S}$.
\end{proof}

We remark that Donaldson's theorem in~\cite{Don05} implies that we can
take the infimum over all metrics in the K\"ahler class, not just the
torus invariant ones. In other words we obtain
\[ \inf_{\omega\in c_1(L)} \Vert S(\omega)-\hat{S}\Vert_{L^2} = \Vert
\Phi\Vert_{L^2},\]
where $L$ is the polarisation that we chose. This shows that existence
of the Calabi flow for all time implies Conjecture~\ref{conj:calinf}
for toric varieties.

Let us also observe that from Equation (\ref{eq:ddtF}) it follows
that if the flow exists for all time, then along a subsequence $u_k$ we
have
\[\int_P u^{ij}_k B_{ij}\,d\mu \to 0.\]
In particular at almost every point where $B$ is strictly concave, we must have
$u^{ij}_k\to 0$. On the other hand
suppose that $B$ is piecewise linear and one of its creases is
parallel to the plane $x_1=0$. This means that $B_{11}$ is a delta
function along that crease, and $B_{ij}$ vanishes for other $i,j$. It
follows that along the subsequence $u_k$ we have $u^{11}_k\to 0$ on 
this crease. In view of the formula (\ref{eq:metricuij}) for the metric
given by $u$ this means that along the creases of $B$ an $S^1$ fibration
collapses. This suggests that the Calabi flow breaks up the toric
variety into the pieces given by the Harder-Narasimhan filtration.  
    
We hope that the calculations here will be useful for showing
that the Calabi flow exists for all time. In particular note that
it follows from Proposition 5.2.2. in~\cite{Don02} that for
$v\in\mathcal{S}$ there is a constant $\lambda > 0$ such that for all
 normalised convex functions $f\in\mathcal{C}_1$ on the polytope we have
\[ \mathcal{L}_{S(v)}(f) \geqslant \lambda\int_{\partial P} f\, d\sigma.\]
Together with
Lemma~\ref{lem:LAdecay} this implies that
for a solution $u_t$ of the Calabi
flow we have a bound of the form
\begin{equation}\label{eq:intdP}
	\int_{\partial P}\tilde{u_t}\,d\sigma \leqslant C(1+t), 
\end{equation}
where $\tilde{u_t}$ is the normalisation of $u_t$. In addition
one would need much better control of
the scalar curvature along the flow in order
to use Donaldson's results (\cite{Don06} and unpublished work in
progress) to control the metrics under the flow at
least in the two dimensional case.

\bibliographystyle{hplain} 
\bibliography{../mybib}
\end{document}